\documentclass[a4paper,11pt]{amsart}

\usepackage{amsmath}
\usepackage{amssymb,amsbsy,amsmath,amsfonts,amssymb,amscd}
\usepackage{latexsym}
\usepackage{graphics}
\usepackage{color}
\usepackage{longtable}
\usepackage{xy}
\usepackage{array}
\usepackage{color}
\usepackage{comment}
\input xy
\xyoption{all}

\newcommand\sC{{\mathcal C}}
\newcommand\sT{{\mathcal T}}

\newcommand\sF{{\mathcal F}}

\newcommand\sZ{{\mathcal Z}}

\newcommand\sW{{\mathcal W}}
\newcommand\sX{{\mathcal X}}

\newcommand\la{\lambda}
\newcommand\Lam{\Lambda}
\newcommand\al{\alpha}

\newcommand\Ga{\Gamma}
\newcommand\De{\Delta}
\newcommand\ga{\gamma}
\newcommand\de{\delta}

\DeclareMathOperator{\Mat}{Mat}

\newcommand{\CC}{\ensuremath{\mathbb{C}}}
\newcommand{\RR}{\ensuremath{\mathbb{R}}}
\newcommand{\ZZ}{\ensuremath{\mathbb{Z}}}

\newcommand{\hol}{\ensuremath{\mathcal{O}}}

\newcommand{\PP}{\ensuremath{\mathbb{P}}}

\newcommand{\ra}{\ensuremath{\rightarrow}}

\def\eea{\end{eqnarray*}}
\def\bea{\begin{eqnarray*}}

\newcommand\dual{\mathrel{\raise3pt\hbox{$\underline{\mathrm{\thinspace d
\thinspace}}$}}}
\newcommand\qe{\ifhmode\unskip\nobreak\fi\quad $\Box$}       % box for QED

\def\BOX{\hfill\lower.5\baselineskip\hbox{$\Box$}}
\newtheorem{theorem}[equation]{Theorem}
\newtheorem{theo}[equation]{Theorem}
\newtheorem{remark}[equation]{Remark}

\newtheorem{defin}[equation]{Definition}

\newtheorem{prop}[equation]{Proposition}

\newtheorem{lemma}[equation]{Lemma}
\newtheorem{example}[equation]{Example}

\newcommand{\sR}{\ensuremath{\mathcal{R}}}

\begin{document}

\title[Orbifolds and Quotients of Tori]{Orbifold Classifying Spaces and Quotients of complex Tori}
\author{ Fabrizio  Catanese }
\address{Lehrstuhl Mathematik VIII, Mathematisches Institut der Universit\"{a}t
Bayreuth, NW II, Universit\"{a}tsstr. 30,
95447 Bayreuth, and  Korea Institute for Advanced Study, Hoegiro 87, Seoul, 
133-722, Korea}
\email{Fabrizio.Catanese@uni-bayreuth.de}

\thanks{AMS Classification:  32Q15, 32Q30, 32Q55, 14K99, 14D99, 20H15, 20K35.\\ 
Key words: Cristallographic groups, group actions on tori, Orbifold fundamental groups, Orbifold classifying spaces.\\ }

\date{\today}

\maketitle

Dedicated to  Gianfranco Casnati, in memoriam, a dear  friend and  a valuable collaborator.

\begin{abstract}
In this paper we 
characterize  the quotients $ X = T/G$
of a complex torus $T$ by the  action of a finite group $G$
as   the K\"ahler orbifold classifying spaces of the even  Euclidean cristallographic groups $\Ga$,
and we prove other similar and stronger  characterizations.
\end{abstract}

\tableofcontents

\section*{Introduction}

Complex tori are the simplest compact K\"ahler manifolds, they are the quotients $T = \CC^n / \Lambda$
where $\Lambda \cong \ZZ^{2n}$ is a discrete subgroup of maximal rank $= 2n$.

To give the flavour of the results of this paper, we observe that complex tori are the cKM: = compact K\"ahler 
Manifolds $X$ which are $K(\ZZ^{2n},1)$ spaces, that is, classifying spaces for the group $\ZZ^{2n}$ (Corollary
82 of \cite{topmethods} shows more generally that they are exactly the compact K\"ahler 
manifolds $X$ which are classifying spaces of a non trivial  Abelian group).
For non experts, this means that the fundamental group $\pi_1(X) \cong \ZZ^{2n}$
and the universal covering $\tilde{X}$ of $X$ is a contractible topological space.

This follows from a stronger result, Prop. 4.8 of \cite{nankai} (= Prop.2.9 of \cite{cat04}) showing  that a compact complex Manifold $X$
 whose  integral cohomology algebra $H^*(X, \ZZ)$  is  isomorphic to the exterior algebra $\Lambda^*(\ZZ^{2n})$ is a  torus if and only if 
 $X$  possesses $n$ independent closed holomorphic 1-forms. If $X$ is projective, then the first property alone suffices to
 guarantee that $X$ is an Abelian variety (a complex tours which is a projective manifold).
 
 The   Generalized Hyperelliptic Manifolds, which we shall  here simply call Hyperelliptic Manifolds,
 are the quotients $X = T/G$ of a complex torus $T$ by the action of  a finite 
group $G$ acting freely (that is, no transformation in the group $G$ has fixed points), and not consisting only of translations.

 In  dimension 2,  these manifolds were introduced and classified by
 Bagnera and de Franchis and Enriques and Severi (\cite{BdF} and  \cite{Enr-Sev}.

 Recall also that  a $K(\Ga, 1)$ manifold is a manifold $M$ such that its universal covering is contractible, and such that $\pi_1(M) \cong \Ga$.
 
 In \cite{cat-corv} it was shown that the 
 Hyperelliptic Manifolds $X$ are the cKM which are classifying spaces for  torsion free  even Euclidean cristallographic groups $\Ga$, thus describing explicitly their 
 Teichm\"uller spaces (Theorem 81 of \cite{topmethods} uses a weaker assumption, similar to the one described above for complex tori, but does not describe the fundamental groups $\Ga$).
 
 The main purpose of this note is to extend these results to quotients $X= T/G$ by groups $G$ which are not acting freely.
 
  In the case where the action of $G$ is {\bf quasi-free}, namely, $G$ acts freely outside of a closed algebraic set
 of codimension at least $2$, this works simply by considering  the normal complex space $X$; but, in the case where the set $\Sigma$ of points $x$ whose stabilizer is nontrivial has codimension $1$, we must replace $X$ by the complex orbifold $\sX$
 consisting of $X$ and of the irreducible divisors $D_1, \dots, D_r$, whose union is the divisorial part of the image of $\Sigma$, each  marked with the integer $m_i$ which is the multiplicity of ramification over $D_i$
 (i.e.,  the ramification index is then $m_i-1$).
 
 As described in different ways in  \cite{d-m}, \cite{isogenous}, \cite{cime}, \cite{topmethods}, \footnote{ \cite{d-m} Section 14, \cite{isogenous} definition 4.4 and Proposition 4.5, pages 25-26, \cite{cime} definition 5.5, 
 Proposition 5.8, pages 101-102, \cite{topmethods} section 6.1, pages 316-318.} we replace the fundamental group
 $\pi_1(X)$ by the Orbifold fundamental group $\Ga: = \pi_1^{orb}(\sX)$, which in the quasi-free case
 is the fundamental group $\pi_1(X^*)$ of the smooth locus $X^*$ of $X$, while in general $\pi_1^{orb}(\sX)$
 can be described as the group of lifts of the transformations of the group $G$ to the universal cover $\tilde{T}$
 of $T$.
 
 In our special case $\Ga$ is a properly discontinuous group of affine transformations of $\CC^n$, 
 a so-called complex cristallographic group (\cite{bieb1,bieb2}, see also \cite{cat-corv}).
 
 We have several options for the assumptions to be made, for instance this is a first result, in the projective case:
 
 \begin{theorem}\label{proj}
 Finite quotients of complex Abelian varieties are: 
 
 (i) the Deligne-Mostow projective orbifolds
 which are orbifold classifying spaces for even Euclidean cristallographic groups $\Ga$,
 
or  more generally 
 
 (ii) the complex  projective orbifolds with KLT singularities 
 which are orbifold classifying spaces for even Euclidean cristallographic groups $\Ga$.
 
 \end{theorem}
 
 In order to deal more generally with quotients of complex tori we need to use a K\"ahler assumption
 (since there are compact complex manifolds diffeomorphic to tori which are not complex tori,
see  \cite{sommese}, based on ideas introduced in \cite{blanchard}).
 
 The concept of a K\"ahler complex space was introduced by Grauert in \cite{grauert}:
 it means that it has a closed real form of type $ (1,1)$
 which at each point is induced from a positive definite one on the Zariski tangent space.

  Fujiki in \cite{fujiki-aut}, see also
 \cite{fujiki-nag} and \cite{ueno}   introduced the concept of a 
 complex space 
 in the class $\sC$, which means that $X$  is dominated by a holomorphic surjective map from a
 K\"ahler space $X'$, equivalently, from a cKM $X'$. It was shown by Varouchas \footnote{Thanks to Thomas Peternell for providing the exact reference of this assertion, stated without further ado in \cite{campana}.} \cite{varouchas1}, \cite{varouchas2},  that this is equivalent to
 requiring that the complex space $X$  is bimeromorphic to a K\"ahler manifold.
 
  \begin{theorem}\label{kahler}
 Finite quotients of complex tori  are: 
 
 (i) the Deligne-Mostow  orbifolds
 which are orbifold classifying spaces for even Euclidean cristallographic groups $\Ga$,
 and  are moreover bimeromorphic  to a  K\"ahler manifold
 
or  more generally 
 
 (ii) the complex   orbifolds with KLT singularities 
 which are orbifold classifying spaces for even Euclidean cristallographic groups $\Ga$,
 and moreover are bimeromorphic  to a  K\"ahler manifold.
 
 \end{theorem}
 
  \begin{theorem}\label{3}
  In Theorems \ref{proj} and \ref{kahler} one can replace the assumption that they are orbifold classifying spaces
  by the conditions:
  
  \begin{enumerate}
  \item
  Their orbifold fundamental group is an  even Euclidean cristallographic group $\Ga$,
  \item
  the integral cohomology algebra of the orbifold covering $Y$ associated to the subgroup $\Lambda < \Ga$ is
  a free exterior algebra $\Lambda^* (\ZZ^{2n})$
   
  \end{enumerate}
  
   \end{theorem}
 
 We were inspired by the recent preprint \cite{gkp}, which established a characterization for the  quasi-free 
 case under the assumptions of homotopy equivalence to such a torus quotient, KLT singularities, and 
 being bimeromorphic  to a  K\"ahler manifold.

  \section{Complex orbifolds, Deligne-Mostow orbifolds, orbifold fundamental groups, orbifold coverings}

\begin{defin} (compare 5.5 in \cite{cime}, and section 4 of \cite{d-m})

 Let  $Z$ be a normal complex space, let $D$ be a closed analytic set
of $Z$ and  let $D_1, \dots, D_r$ be the irreducible components of $D$ of codimension $1$ in the case where $D$ is compact, else we can also allow the set $\{ D_i\}$ to be infinite (but shall keep the former notation for the sake of a more readable notation).

(1) Attaching  to each $D_i$ a positive integer $m_i \geq 1$, we obtain a {\bf complex orbifold}, if moreover
$ D =  D_1 \cup \dots \cup  D_r \cup Sing(Z)$.

(2) The {\bf orbifold fundamental group} $\pi_1^{orb} (Z \setminus D, (m_1, \dots, m_r))$ is defined as the quotient 
$$ \pi_1^{orb} (Z \setminus D, (m_1, \dots, m_r) : = \pi_1 (Z \setminus D) / \langle \langle(\ga_1^{m_1}, \dots, \ga_r^{m_r} \rangle \rangle$$

of the fundamental group of $(Z \setminus D)$ by the subgroup normally generated by simple geometric loops $\ga_i$ going each around a smooth point of the divisor $D_i$ (and counterclockwise). 

(3) The orbifold is said to be {\bf quasi-smooth} or geometric  if moreover $D_i$ is smooth outside of $Sing(Z)$.

(4) The orbifold is said to be a {\bf Deligne-Mostow orbifold} if moreover for each point $z \in Z$
there exists a local chart $\phi : \Omega  \ra U = \Omega/G$, where $0 \in \Omega \subset \CC^n$,
$G$ is a finite subgroup of $GL(n, \CC)$, $\phi(0) = z$, $U$ is an open neighbourhood of $z$, and
the orbifold structure is induced by the quotient map. That is, $D \cap U$ is the branch locus of $\Phi$,
and the integers $m_i$ are the ramification multiplicities. 

(5) We identify two orbifolds under the equivalence relation generated by forgetting the divisors $D_i$ with multiplicity $1$.

\end{defin}

\begin{remark}
(i) A D-M (= Deligne-Mostow) orbifold is quasi-smooth, and then $Z$ has only quotient singularities, which are rational singularities.

(ii) In the case where there is no divisorial part, and we have an orbifold, then the orbifold fundamental group
is the fundamental group of $Z \setminus Sing (Z)$.

(iii) If $Z = M/ \Ga$ is the quotient of a complex manifold $M$ by a properly discontinuous subgroup $\Ga$, then
$Z$ is a D-M orbifold, by Cartan's lemma (\cite{cartan}), saying that the action of the stabilizer sungroup
becomes linear after a local change of coordinates.

(iv) one could  more generally consider a wider class of orbifolds allowing  also the multiplicity $m_i = \infty$: this means that the relation $\ga_i^{m_i}=1$
is a void condition.

(v) replacing  $D$ by its  intersection with the union of the singular locus with  the divisorial components does not change the
orbifold  fundamental group.

\end{remark}

Thanks to the extension by Grauert and Remmert \cite{g-r} of Riemann's existence theorem to finite holomorphic maps
of normal complex spaces, we have that to a subgroup of the orbifold fundamental group corresponds 
an {\bf orbifold covering} of orbifolds, that is (see for instance  \cite{d-m})

\begin{itemize}
\item
a finite holomorphic map 
$$ f : \sZ : = (Z, D (m_1, \dots, m_r)) \ra \sW : = (W,B (n_1, \dots, n_s))$$ such that
\item
$f$ induces an \'etale (unramified) map $F : Z \setminus D \ra W \setminus B$,
\item
for each $Z_i$, $ f (Z_i) = B_j$ for some $j$, and locally $\ga_i \mapsto \de_j^{a_i}$,
where $ n_j = a_i m_i$, and $\de_j$ is a simple loop around $B_j$.
\item
$f^{-1}(B_j) $ is set theoretically a union of divisors $D_i$ (keep in mind here the equivalence relation explained in (5)).
\end{itemize} 

To the trivial subgroup corresponds the orbifold universal cover 
$$( \tilde{Z}, \tilde{D}, \{ \tilde{m}_j \}).$$

\begin{defin}
We say that an  orbifold $(Z, D (m_1, \dots, m_r))$ is an orbifold classifying space if its 
universal covering $( \tilde{Z}, \tilde{D}, \{ \tilde{m}_j \})$
satisfies the properties 

(a) either $\tilde{Z}$ is contractible and the multiplicities $\tilde{m}_j $ are all equal $1$, or

(b) there is a homotopy retraction of $\tilde{Z}$ to a point which preserves the subdivisor $\tilde{D}'$
consisting of the irreducible components with multiplicity $\tilde{m}_j >1$.
\end{defin} 

\begin{remark}

We end  now  this section showing two simple   examples of an orbifold which is not a Deligne-Mostow orbifold.

(1) We just take as orbifold space $Z= \CC^2$ and as divisors three distinct  lines $L_1, L_2, L_3$ through the origin,
marked with multiplicities $m_1, m_2, m_3$. 

If this orbifold were an orbifold $ \CC^2 / G$, since  $\CC^2$ is simply connected, then $G$ would be isomorphic to
the orbifold fundamental group $\pi$ of $Z$.

But $\pi$ (see for instance page 140 of \cite{locfund}) has a presentation

$$ \pi : = \langle \ga_0,  \ga_1,  \ga_2,  \ga_3 | [\ga_0, \ga_i]=1,{\rm for} \  i=1,2,3, \ga_0 = \ga_1  \ga_2  \ga_3 \rangle.$$

Since $\pi \cong G$ is finite, also $\pi /  \langle \ga_0  \rangle$ is finite. However this quotient is finite if and only if
$$ \frac{1}{m_1} + \frac{1}{m_2} + \frac{1}{m_3} > 1,$$
since the corresponding orbifold covering of $\PP^1$ branched in three points must be simply connected,
hence equal to $\PP^1$; and this is equivalent to the above inequality which is only satisfied for
the Platonic triples $(2,2,n), (2,3,3),(2,3,4),(2,3,5)$.

(2) Another easy example is given by a normal surface singularity which is not a quotient singularity, for instance an elliptic surface singularity.
\end{remark}
\section{Euclidean cristallographic groups and Actions of a finite group $G$ on a complex torus $T$}

For the reader's benefit, we repeat some results on complex cristallo graphic groups, as exposed in \cite{cat-corv}.

\begin{defin}\label{cristall}
(i) A group $\Ga$ is an abstract Euclidean cristallographic group if there exists an exact
sequence of groups
$$ (*) \ \ 0 \ra \Lam \ra \Ga \ra G \ra 1 $$
such that
\begin{enumerate}
\item
$G$ is a  finite group
\item
$\Lam$ is free abelian (we shall denote  its rank by $r$)
\item
Inner conjugation $ Ad : \Ga \ra  Aut (\Lam) $ has Kernel exactly $\Lam$,
hence $ Ad$  induces an embedding, called {\bf Linear part},
$$ L : G \ra GL (\Lam) : = Aut (\Lam) .$$
\end{enumerate}

(ii) An {\bf affine realization defined over a field $ K \supset \ZZ$ }  of an abstract Euclidean cristallographic group $\Ga$  is a homomorphism
(necessarily injective) 
$$\rho : \Ga \ra Aff (\Lam \otimes_{\ZZ} K)$$ such that 

[1] $\Lam$ acts by translations on $ V_K := \Lam \otimes_{\ZZ} K$,  $ \rho(\la) (v) =   v + \la$,

[2]  for any $\ga$ a lift of $g \in G$ we have:
$$  V_K \ni v \mapsto \rho(\ga) (v) =  Ad (\ga) v + u_{\ga} = L(g) v + u_{\ga}, \ {\rm for \ some } \ \ u_{\ga} \in  V_K.$$ 
\end{defin}

For the following Theorems, see \cite{cat-corv}: observe that the unicity of the affine realization was proven by Bieberbach in 1912 (\cite{bieb2})(who proved also many other deeper results).

\begin{theo}\label{affinereal}
Given an abstract Euclidean cristallographic group there exists an affine realization, for each    field $K \supset \ZZ$,
and its  class is unique.

\end{theo}

Assume that we have the  action of a finite group $G'$ on  a complex torus 
 $ T= V / \Lam'$, where $V$ is a complex vector space, and $ \Lam' \otimes_{\ZZ} \RR \cong V$.
 
   Since every holomorphic map between complex tori lifts to a complex  affine map
 of the respective universal covers, we can attach  to the group $G'$  the group $\Ga$  of (complex) affine transformations
 of $V$ which are  lifts of  transformations of the group $G'$.

 Again one easily sees (\cite{cat-corv}:
 \begin{prop}\label{notranslations}
 $\Ga$ is an Euclidean cristallographic group, via the  exact sequence 
 $$0 \ra \Lam \ra \Ga \ra G \ra 1  $$
where  $\Lam  \supset  \Lam'$  is the  lattice in $V$ such that  $\Lam: = Ker (Ad), \ Ad : \Ga  \ra GL  (\Lam')$,
and  $ G \subset Aut (V/ \Lam)$
 contains no translations.
  \end{prop}
 
 Hence the datum of  the action of a finite group  $G$  on a complex torus $T$, containing no translations, is equivalent to giving:

\begin{enumerate}
\item
a cristallographic group $\Ga$
\item
 a complex structure $J$ on the real vector space $V_{\RR}$ which makes the action of $G$ complex  linear. 
\end{enumerate}

The complex structure $J$ exists if and only if $\Ga$ is even, according to the following definition:

\begin{defin}\label{even}
A cristallographic group $\Ga$ is said to be {\bf even} if:

i)  $\Lam$ is a free abelian group  of even rank $ r = 2n$

ii) considering the associated   faithful representation $$ G \ra Aut (\Lam \otimes \CC),$$  for each real irreducible representation $\chi$ of $G$, 
 {\bf $ (\Lam \otimes \CC)_{ \chi}$ has even complex  dimension}.

\end{defin}

 \section{Proof of the Main Theorems}
 
\subsection{Properties of $X=T/G$ and proof of the easier implications}

I)  Let $X = T/G$ be as above the quotient orbifold of a torus by the action of a finite group.
Since the universal covering of $T = V / \Lam$ is the vector space $V$, which is contractible,
 and $ X = V / \Ga $, where $\Ga : = \pi_1^{orb} (X) $  we obtain that $X$ is an orbifold classifying space.
Moreover, that  $\Ga$ is an Euclidean cristallographic group follows from proposition \ref{notranslations}, 
 and $\Ga$ is even since there is a $G$-invariant complex structure.

II) Consider now  the normal subgroup $G^{pr} <G$ generated by the pseudoreflections, that is, the linear maps which have
only one eigenvalue $\neq 1$. 

We have in particular a factorization
$$ T \ra  W: = T/ G^{pr} \ra X = T/G, $$
where the second map is quasi-\'etale, that is, ramified only in codimension at least $2$.

At any point $ t \in T$ having a nontrivial stabilizer $G_t < G$, we have similarly  a corresponding normal subgroup
$G_t^{pr}$, generated by the pseudoreflections in $G_t$. By Chevalley's Theorem, the local quotient
of $T$ at $t$ by $G_t^{pr}$ is smooth, and then we have the further quotient by $G_t / G_t^{pr}$.

In particular, $X$ is a Deligne-Mostow orbifold and its singularities are quotient singularities.

By \cite{k-m} (Prop. 5.15, page 158)  quotient singularities $(X,x)$ are rational singularities, that is, they are normal and, if $f : Z \ra X$ is a local resolution, 
then $\sR^if_* \hol_Z=0 $ for $i \geq 1$. They enjoy also the stronger property of being KLT (Kawamata Log Terminal)
singularities.

Indeed Prop. 5.22 of \cite{k-m} (where dlt=KLT if there is no boundary divisor $\De, \De'$) says that KLT singularities are rational singularities, while Prop. 5.20, page 160, says that if we have a finite morphism between normal varieties,
$ F : T \ra X$, then {\bf $X$ is KLT if and only if $T$ is KLT)}.

III) If $T$ is projective, then also $X$ is projective, since, by averaging,  we can find on $T$ a $G$-invariant very ample divisor.

\subsection{The K\"ahler property}
  IV) In the general case where the torus $T$ is not projective, but only K\"ahler, 
   to show that $X$ enjoys some  K\"ahlerian properties, we need to recall some results
  by Fujiki and others (which could also be used to slightly vary the hypotheses in our results).
  
  As already mentioned in the Introduction, Grauert \cite{grauert} defined the 
concept of a K\"ahler complex space, and later Varouchas \cite{varouchas1} proved that if $X \ra Y$ is a surjective
holomorphic map with $Y$ reduced, and $X$ K\"ahler, then $Y$ is bimeromorphic to a K\"ahler manifold.

  Instead  Fujiki in \cite{fujiki-aut} (see also \cite{fujiki-nag} and \cite{ueno}),  introduced the concept of a complex space $X$  bimeromorphic to a K\"ahler manifold, and of a 
 complex space 
 in the class $\sC$, meaning  that $X$  is dominated by a holomorphic surjective map from a
 K\"ahler space $X'$, equivalently, from a cKM $X'$.
  
 Here the first  basic   result in this line of thought:
 1.5 of \cite{ueno} implies that: 
 
 {\bf given  a generally  finite map $ Y \ra X$, $X$ is in $\sC$ if and only if $Y \in \sC$}
 
   (in \cite{catAV}, sections 17-1.9 was observed the easier result that if a compact complex manifold $Z$ has a generically finite map to a cKM $M$, then $Z$ is bimeromorphic to a K\"ahler manifold: this evidently also holds if $Z$ is a complex space).
 
 Fujiki also asked (remark 4.4 , page 35 of \cite{fujiki-rims}) whether manifolds in the class $\sC$ are just those which are bimeromorphic to a K\"ahler manifold: his conjecture was shown  to be true by Varouchas \cite{varouchas1},  \cite{varouchas2}. 
 
 Whence, 
 
{\bf   (I) the quotients $X=T/G$ are bimeromorphic to a K\"ahler manifold.
 
Moreover   

(II) the quotients $X=T/G$ are also K\"ahler complex spaces}  since  \cite{varouchas2} proved that if $W \ra X$
is proper and open, with $X$ normal, from the property that $W$ is a K\"ahler space follows that also $X$ is K\"ahler.

 Since, again for instance by \cite{varouchas1}, property (II) implies property (I), in  our Theorems we opt for
 assuming only Property (I). 
 
 Finally, a crucial result is that (see Prop. 1.3 of \cite{ueno}) if a compact complex  manifold $M$  is in the class $\sC$,
 then the cohomology of $M$ admits a Hodge decomposition, and in particular every holomorphic form is d-closed, and 
 there is an Albanese map $ \al : M \ra Alb(M)$, such that $\al^* : H^1 (Alb(M),\CC) \ra  H^1 (M,\CC)$
 is an isomorphism.

 \subsection{Proof of Theorem \ref{proj}.} We need to show the converse implication.
 {\bf Key argument: we consider the orbifold covering $Y$ associated to the normal subgroup 
 $$ \Lam < \Ga : = \pi_1^{orb}(X),$$
 and we show that $Y$ is a complex torus (here projective because $X$ is projective and $Y$ 
 has a finite map to $X$).}
 
 \begin{lemma}
 The orbifold $Y$ is just a normal complex space, that is,  there are no marked divisors with  multiplicity $m_i \geq 2$.
 \end{lemma}
 
 \begin{proof}
 Consider the exact sequence
 $$ 1 \ra \Lam \ra \Ga \ra G \ra 1.$$
 
 Then the generators $\ga_i$ have finite order $m_i$, hence their image in $G$ has order exactly $m_i$, because
 $\Lam$ is torsion free.
 
 This means that the covering $ Y \ra X$ is ramified with multiplicity $m_i$ at the divisor $D_i$, and therefore their inverse image
 in $Y$ is a reduced divisor with multiplicity $1$.
 
 \end{proof}
 
 In case (i), since $X$ is a   Deligne-Mostow orbifold, then also $Y$ is a D-M orbifold, hence it is a normal space with  quotient singularities, and these are rational singularities.
 
 Let $Y'$ be a resolution of $Y$. Since $Y$ has rational singularities, $\sR^1f_* (\ZZ_{Y'})=0$ and we have an isomorphism
 $$ H^1(Y', \ZZ) \cong H^1(Y, \ZZ) \cong \ZZ^{2n}.$$
 
 Hence  the Albanese variety of $Y'$ is a complex torus of dimension $n$, and the Albanese map
 $\al' : Y' \ra T: = Alb(Y')$ factors through $Y$, and $\al : Y \ra T$ is a homotopy equivalence,
 in particular it has degree $1$ (because it induces an isomorphism of $H^{2n}(T, \ZZ) \cong H^{2n}(Y, \ZZ)$).
 
 We follow a similar argument to the one used in \cite{nankai}, proof of Proposition 4.8: it suffices to show 
 that $\al$ is  finite, because $\al$,  being finite and birational,  is then an isomorphism $ Y \cong T$ by normality.

  Now, since $\al$ is birational, by Zariski's Main Theorem (the Hartogs property in the case of normal complex spaces) $\al$ is an isomorphism unless there is a divisor $D$ which is contracted 
  by $\al$. And, since $H^{j}(T, \ZZ) \cong H^{j}(Y, \ZZ)$, the class of $D$ is trivial in $H^{2}(Y, \ZZ)$, and a fortiori
  its pull back $D'$ to $Y'$ is trivial.
  
 This is a contradiction since, $Y'$ being K\"ahler, 
 the class of $D'$ cannot be trivial.
 
  In case (ii) the proof is identical, we need only to   establish that $Y$ has rational singularities.
  
  As already discussed, if $X$ has KLT singularities, by Prop. 5.20 of \cite{k-m} also $Y$ has KLT 
  singularities, which are rational singularities.
  
  \subsection{Proof of Theorems \ref{kahler} and \ref{3}}The proof is essentially the same.
  
  Indeed, $X$ is assumed to be bimeromorphic to a K\"ahler manifold, that is, in the class $\sC$, and by the results of Fujiki, Ueno and Varouchas also $Y$ is bimeromorphic to a K\"ahler manifold $Y'$, that we can assume to dominate $Y$.
  
  By our assumption $Y$ has again rational singularities, and we may
   consider again the Albanese map $\al' : Y' \ra T : = Alb (Y')$, which again factors through a birational map
   $\al : Y \ra T$. We derive the same contradiction.

  \begin{remark}
  One may ask whether one can replace the condition of KLT singularities for $X$ by the condition
  that $X$ has rational singularities, proving  then that also $Y$ has rational singularities.
  \end{remark}

 \section{Parametrizing Families}
 
 We simply observe now, as in \cite{cat-corv},  how these orbifolds are parametrized by a finite union of  connected complex manifolds,
  which are just 
    products of Grassmann manifolds.
 
 The  connected component $\sT_n$ of the Teichm\"uller space of  $n$-dimensional complex tori  (see  \cite{nankai}, 
\cite{cat04} and \cite{handbook} )
is  the open set $\sT_n$ of the complex
Grassmann Manifold $Gr(n,2n)$, image of the open set of matrices

$\sF : = \{ \Omega \in \Mat(2n,n; \CC) \ | \ i^n det  (\Omega \overline
{\Omega}) > 0 \}.$

Over $\sF$ lies the following tautological family of complex tori:    consider
a fixed lattice $ \Lam  : =  \ZZ^{2n}$, and associate
to each matrix $ \Omega $ as above   the subspace $V$ of $\CC^{2n} \cong \Lam \otimes \CC$ given as
$$ V : =  \Omega \CC^{n},$$ so that
$ V \in Gr(n,2n)$ and $\Lam \otimes \CC \cong V \oplus \bar{V}.$

To  $ V $ we associate then the  torus 
$$T_V : = V / p_V (\Lam) = (\Lam \otimes \CC)/ (\Lam \oplus  \bar{V} ),$$
$p_V : V \oplus
\bar{V} \ra V$ being the projection onto the first summand.

The crystallographic group $\Ga$ determines an action of $G \subset SL (2n, \ZZ)$ on $\sF$ and on $\sT_n$,
obtained by multiplying the matrix  $\Omega$ with matrices $g \in G$ on the right.

We define then $\sT_n^G$ as the locus of fixed points for the action of $G$. If $V \in \sT_n^G$, then 
$G$ acts as a group of biholomorphisms of $T_V$, and we associate then to such a $V$ the orbifold
$$ X_V : =  T_V / G.$$

  We see as in \cite{cat-corv} that  $\sT_n^G$ consists of  a finite number of  components, indexed by the Hodge type of the Hodge decomposition.

 \bigskip

%%%%%%%%%%%%%%%%%%%%%%%%%%%%%%%%%%%%%%%%%%%%%%%%%%%%%%


\begin{thebibliography}{Grif-SchmX}
%%%%%%%%%%%%%%%%%%%%%%%%%%%%%%%%%%%%%%%%%%%%%%%%%%%%%%



\bibitem[BdF08]{BdF}
Bagnera, G.; de Franchis, M.
{\em Le superficie algebriche le quali ammettono una rappresentazione parametrica mediante funzioni iperellittiche di due argomenti. }
Mem. di Mat. e di Fis. Soc. It. Sc. (3) 15, 253--343 (1908).

 \bibitem[BCF14]{bcf}
 Bauer, I.,  Catanese, F., Frapporti, D.
{\em  Generalized Burniat type surfaces and Bagnera-de Franchis varieties}, J. Math. Sci. Univ. Tokyo {\bf 22} (2015), 55-111.

\bibitem[Bieb11]{bieb1}
Bieberbach, L.
{\"Uber die Bewegungsgruppen der euklidischen R\"aume. (Erste Abhandlung.) }
Math. Ann. 70, 297--336 (1911).

\bibitem[Bieb12]{bieb2}
Bieberbach, L.
{\"Uber die Bewegungsgruppen der euklidischen R\"aume. (Zweite Abhandlung.) Die Gruppen mit einem endlichen Fundamentalbereich. }
Math. Ann. 72, 400--412 (1912).

\bibitem[Blan56]{blanchard}
Blanchard, Andr\'e
{\em Sur les vari\'et\'es analytiques complexes. }
Ann. Sci. \'Ec. Norm. Sup\'er., III. S\'er. 73, 157--202 (1956).

\bibitem[Cam91]{campana}
Campana, Frederic
{\em On twistor spaces of the class $\sC$.}
J. Differ. Geom. 33, No. 2, 541--549 (1991).

\bibitem[Cart57]{cartan}
Cartan, Henri
{\em Quotient d?un espace analytique par un groupe d'automorphismes. }
Princeton Math. Ser. 12, 90--102 (1957).

 \bibitem[Cat95]{catAV}
     Catanese, Fabrizio 
     {\em Compact complex manifolds bimeromorphic to tori. }, in `Abelian varieties' (Egloffstein, 1993), 55--62, de Gruyter, Berlin, (1995).
    
    \bibitem[Cat00]{isogenous} 
     Catanese, Fabrizio
 {\em Fibred surfaces, varieties isogenous to a product and related moduli spaces. }
Am. J. Math. 122, No. 1, 1--44 (2000).

\bibitem[Cat02]{nankai}
Catanese, Fabrizio 
{\em Deformation types of real and complex manifolds.}
 Contemporary trends in algebraic geometry and algebraic topology (Tianjin, 2000), 195--238, Nankai Tracts Math., 5, World Sci. Publ., River Edge, NJ,(2002).

\bibitem[Cat04]{cat04}
Catanese, Fabrizio 
{\em Deformation in the large of some complex manifolds, I} Ann. Mat. Pura Appl. (4)
183,  Volume in Memory of Fabio Bardelli, (2004), no. 3, 261--289.

\bibitem[Cat06]{locfund}
Catanese, Fabrizio
{\em Surface classification and local and global fundamental groups. }
Atti Accad. Naz. Lincei, Cl. Sci. Fis. Mat. Nat., IX. Ser., Rend. Lincei, Mat. Appl. 17, No. 2, 135--153 (2006).

\bibitem[Cat08]{cime}
Catanese, Fabrizio
{\em Differentiable and deformation type of algebraic surfaces, real and symplectic structures. }
Catanese, Fabrizio (ed.) et al., Symplectic 4--manifolds and algebraic surfaces. Lectures given at the C.I.M.E. summer school, Cetraro, Italy, September 2--10, 2003. Berlin: Springer; Florence: Fondazione C.I.M.E  Lecture Notes in Mathematics 1938, 55--167 (2008).



\bibitem[Cat13]{handbook}
Catanese, Fabrizio,
{\em A superficial working guide to deformations and moduli. } 
Farkas, Gavril (ed.) et al., Handbook of moduli. Volume I. Somerville, MA: International Press; Beijing: Higher Education Press. Advanced Lectures in Mathematics (ALM) 24, 161--215 (2013).

\bibitem[Cat15]{topmethods}
Catanese,Fabrizio,
{\em  Topological methods in moduli theory. } Bull. Math. Sci. 5 (2015), no. 3, 287-- 449.

 \bibitem[Cat-Cor17]{cat-corv}  
Catanese, Fabrizio; Corvaja, Pietro,
{\em Teichm\"uller spaces of generalized hyperelliptic manifolds. } 
Angella, Daniele (ed.) et al., Complex and symplectic geometry. Based on the presentations at the INdAM meeting "Complex and symplectic geometry", Cortona, Italy, June 12--18, 2016. Cham: Springer  Springer INdAM Series 21, 39--49 (2017). 

\bibitem[Del-Mos93]{d-m}
Deligne, Pierre; Mostow, George Daniel
{\em Commensurabilities among lattices in $\PP U(1,n)$}.
Annals of Mathematics Studies. 132. Princeton, NJ: Princeton University Press. 183 p. (1993)

\bibitem[ES09]{Enr-Sev}
Enriques, F.; Severi, F..
{\em M\'emoire sur les surfaces hyperelliptiques. }  
Acta Math. 32, 283-- 392 (1909) and 
33, 321--403 (1910).


 \bibitem[Fuj78]{fujiki-aut}
Fujiki, Akira
{\em On automorphism groups of compact K\"ahler manifolds. }
Invent. Math. 44, 225--258 (1978).

 \bibitem[Fuj78-b]{fujiki-rims}
Fujiki, Akira
{\em Closedness of the Douady spaces of compact K\"ahler spaces. }
Publ. Res. Inst. Math. Sci., Kyoto Univ. 14, 1--52 (1978).

 \bibitem[Fuj82]{fujiki-nag}
Fujiki, Akira
{\em On the Douady space of a compact complex space in the category $\sC$. }
Nagoya Math. J. 85, 189--211 (1982).

\bibitem[Gra62]{grauert}
Grauert, Hans
{\em \"Uber Modifikationen und exzeptionelle analytische Mengen. }
Math. Ann. 146, 331--368 (1962).

\bibitem[G-R58]{g-r}
Grauert, Hans; Remmert, Reinhold
{\em Komplexe R\"aume.}
Math. Ann. 136, 245--318 (1958).

\bibitem[GKP23]{gkp}
Greb, Daniel; Kebekus, Stefan; Peternell, Thomas.
{\em Miyaoka-Yau inequalities and the topological characterization of certain klt varieties}
arXiv:2309.14121, to appear in a Comptes Rend. Ac. Sci. volume in memory of jean Pierre Demailly. 

\bibitem[K-M98]{k-m}
Koll\'ar, J\'anos; Mori, Shigefumi
{\em Birational geometry of algebraic varieties. With the collaboration of C. H. Clemens and A. Corti. }
Cambridge Tracts in Mathematics. 134. Cambridge: Cambridge University Press. viii, 254 p. (1998).


 \bibitem[Lan01]{Lange}
  Lange, Herbert
{\em Hyperelliptic varieties. }
Tohoku Math. J. (2) 53 (2001), no. 4, 491--510. 

\bibitem[Som75]{sommese}
Sommese, Andrew John
{\em Quaternionic manifolds. }
Math. Ann. 212, 191--214 (1975).

\bibitem[UY76]{U-Y}
Uchida, Koji; Yoshihara, Hisao
{\em Discontinuous groups of affine transformations of $ \CC^3$}. 
Tohoku Math. J. (2) 28,  no. 1, 89--94 (1976).

\bibitem[Ueno83]{ueno}
Ueno, Kenji
{\em Introduction to the theory of compact complex spaces in the class $\sC$.}
Algebraic varieties and analytic varieties, Proc. Symp., Tokyo 1981, Adv. Stud. Pure Math. 1, 219--230 (1983).


\bibitem[Var86]{varouchas1}
Varouchas, Jean
{\em Sur l'image d'une vari\'et\'e K\"ahl\'erienne compacte. (On the image of a compact K\"ahler manifold). }
Fonctions de plusieurs variables complexes V, S\'emin. F. Norguet, Paris 1979-1985, Lect. Notes Math. 1188, 245--259 (1986).

\bibitem[Var89]{varouchas2}
Varouchas, Jean
{\em K\"ahler spaces and proper open morphisms. }
Math. Ann. 283, No. 1, 13--52 (1989).

\end{thebibliography}
\end{document}